\documentclass[11pt]{amsart}
 \usepackage{amsopn}
 \usepackage{amsmath,amsthm,amssymb}

\usepackage{amssymb}
\usepackage{graphicx}
\usepackage{amsmath}
\usepackage[latin1]{inputenc}
\usepackage{amsfonts}

 \theoremstyle{plain}
 \newtheorem{thm}{Theorem}[section]
 \newtheorem{theor}{Theorem}

 \newtheorem{lemma}[thm]{Lemma}
  \newtheorem{rem}[thm]{Remark}
  \theoremstyle{definition}

 \theoremstyle{remark}

\newcommand{\rr}{\mathbb{R}}
\newcommand{\loren}{\mathbb{L}}
\newcommand{\pp}{\varphi}

\newcommand{\en}{\nu_{p}M}
\newcommand{\pe}{\tau^{\mathcal{N}}}
\newcommand{\pae}{\widetilde{\tau}}

\newcommand{\meti}{\left\langle}
\newcommand{\metd}{\right\rangle}

\newcommand{\esf}{\mathbb{S}}
\newcommand{\hip}{\mathbb{H}}
\newcommand{\hor}{\mathcal{H}}

\newcommand{\aff}{\mathbb{A}}
\newcommand{\an}{\mathbb{A}^{n}}
\newcommand{\hol}{\widetilde{\Phi}}

\newcommand{\wnu}{\widetilde{\nu}}
\newcommand{\ws}{\widetilde{\Sigma}}
\newcommand{\nul}{\mathcal{N}}
\newcommand{\pr}{\mathbf{pr}}
\newcommand{\oo}{\Phi^{0}_{r}\cdot}
\newcommand{\dd}{\mathcal{D}}

\begin{document}
 \title[On the nullity distribution of a submanifold of a space form] {On the nullity distribution
 of a submanifold of a space form}
\author[Francisco Vittone]{Francisco Vittone}

\begin{abstract}If $M$ is a submanifold of a space form, the nullity distribution $\nul$ of its second fundamental form is (when defined) the common kernel of its shape operators. In this paper we will give a local description of any submanifold of the Euclidean space by means of its nullity distribution. We will also show the following global result: if $M$ is a complete, irreducible submanifold of the Euclidean space or the sphere then its complementary distribution $\nul^{\bot}$ is completely non integrable. This means that any two points in $M$ can be joined by a curve everywhere perpendicular to $\nul$. We will finally show that this statement is false for a submanifold of the hyperbolic space. 

\end{abstract}

 \maketitle
 \section{Introduction}
Let $M^{n}$ be a submanifold of a space form $Q^{n+k}$. The nullity distribution $\nul$ of the second fundamental form  of $M$ is defined as the common kernel of the shape operators of $M$. The index of nullity (or relative nullity) of $M$ at $p$ is the dimension of $\nul_{p}$. It is well known, from Codazzi equation, that $\nul$ is an autoparallel distribution restricted to the open and dense subset of  $M$ where the index of nullity is locally constant. Moreover, its integral manifolds are totally geodesic submanifolds of the ambient space $Q^{n+k}$. 
If $M$ is complete, and one restricts to the open subset $U$ of points of $M$ where the index of nullity is minimal, then the integral manifolds of $\nul$ through points of $U$ are also complete (see for instance \cite[Ch. 5]{daj} or \cite{ferus}). 

Observe that for an Euclidean submanifold, the nullity distribution arises naturally in many problems since  $\nul=\ker(dG)$, where $G:M\to G_{k}(\rr^{n+k})$ is the Gauss map. Moreover, there are many examples of complete submanifolds of the Euclidean space with constant index of nullity and some of them can be obtained from  the so called Gauss parametrizations, which involve the Gauss map (see \cite{dajgrom}). 

In this paper we will study the perpendicular distribution $\hor:=(\nul)^{\bot}$. 

For a general distribution $\mathcal{D}$ on a Riemannian manifold, it is an interesting problem to decide whether it is completely non integrable, in the sense that any two points can be joined by a curve always tangent to $\mathcal{D}$ (this means that the Carnot-Caratheodory distance, associated to $\mathcal{D}$,  is finite \cite{gromov}). 
For the case of submanifolds, a very important result about completely non-integrability, is the so called \textit{Homogeneous Slice Theorem} \cite{hot}, which has many important applications (see \cite{obc}, \cite{tho}). It states, in particular, that for an irreducible complete (connected) isoparametric submanifold of the sphere, one can join any two points by a curve always perpendicular to any given eigendistribution of its shape operator. 

Our main goal is to prove a general global result about the non-integrability of the distribution $\hor$ for submanifolds of the Euclidean space or the sphere. Namely,

\begin{theor}
Let M be an immersed, complete, irreducible submanifold of the Euclidean space or the sphere with constant index of nullity. Then any two points of $M$ can be joined by a curve always perpendicular to the nullity distribution.
\label{globalthm}
\end{theor}
Observe that, in general, the canonical projection $\pr:M\to M/\nul$, where $M/\nul$ is the set of all maximal integral manifolds of $\nul$, is not a Riemannian submersion. 

If we drop the assumption that the index of nullity is constant, then Theorem \ref{globalthm} still holds on any connected component of the open subset where this index is minimal. On the other hand, the assumption of completeness can not be removed. In fact, there are abundant many local counter-examples of the above theorem obtained as the union of suitable parallel manifolds. Moreover, Theorem \ref{localth} shows that any local counter-example arises in this way.

Theorem \ref{globalthm} is not true for submanifolds of the hyperbolic space (see remark \ref{hyp}).

For infinite dimensional isoparametric Hilbert submanifolds of codimension at least two, a similar result to Theorem \ref{globalthm} was proved by Heinzte and Liu in \cite{Heintze}, using strongly the isoparamatricity condition. This was a crucial step in the proof of the homogeneity of this type of submanifolds \cite{Heintze}. Theorem \ref{globalthm} shows that, in finite dimension, the non-integrability of the distribution $\nul^{\bot}$ is a very general global fact, that does not depend on any extra property of the submanifold. 

\section{Basic definitions and general properties}

\subsection{The nullity distribution}
Let $M$ be an $n$-dimensional (immersed) Riemannian submanifold of a (simply connected) space form, i.e., the Euclidean space $\rr^{n+k}$, the sphere $\esf^{n+k}$ or the hyperbolic space $\hip^{n+k}$. We denote by $\nabla$ the Levi-Civita connection of $M$  and by $\widetilde{\nabla}$ the Levi-Civita connection of the ambient space form. We will always denote by $\nu M:=TM^{\bot}$ the normal bundle of $M$ endowed with the normal connection $\nabla^{\bot}$.
The second fundamental form of $M$ will be denoted by $\alpha$ and the shape operator by $A$. These two tensors are related by the well known formula $$\meti\alpha(X,Y),\xi\metd=\meti A_{\xi}X,Y\metd,$$ which is symmetric in $X$ and $Y$, for any $X,Y$ tangent fields and $\xi$ normal field. 

The connection $\nabla\oplus\nabla^{\bot}$ in $TM\oplus\nu M$ will be denoted by $\overline{\nabla}$. The  Codazzi equation states that $\meti (\overline{\nabla}_{X}A)_{\xi}Y,Z\metd$, or equivalently $(\overline{\nabla}_{X}\alpha)(Y,Z)$, is symmetric in $X,Y,Z$.

The  \textsl{nullity subspace} of $M$ at $p$ is the subspace $\nul_{p}$ of $T_{p}M$ defined by 
$$\nul_{p}:=\{x\in T_{p}M:\alpha(x,\cdot)\equiv 0\}=\bigcap_{\xi\in\en}\ker(A_{\xi}).$$ 
If $M\subset \rr^{n+k}$ then $\nul_{p}$ coincides with $\ker(d \mathbf{G}_{p})$
where $\mathbf{G}:M\to G_{k}(\rr^{n+k})$ is the Gauss map that assigns to each point $p\in M$ its normal space $\en$.

Let $\mu(p):=\dim(\nul_{p})$, which is called the \textsl{index of nullity} of $M$ at $p$. Let $\mathcal{C}$ be the set of points $p\in M$ such that $\mu$ is constant in a neighborhood of $p$. Then it is standard to prove, since $\mu$ can not increase locally, that $\mathcal{C}$ is an open and dense subset of $M$ and so $\mu$ is constant on each connected component of $\mathcal{C}$. Moreover, from Codazzi equation, it follows that $\nul:p\to\nul_{p}$ defines a $C^{\infty}$ autoparallel  distribution (and hence with totally geodesic integral manifolds) on each connected component of $\mathcal{C}$. It is not difficult to see that its integral manifolds are also totally geodesic in the ambient space form. 
If $M$ is complete, then any (maximal) integral manifold of  $\nul$ through a point with minimal index of nullity is complete (see \cite{ferus}). 
Then

\begin{lemma} Let $M$ be a complete Riemannian submanifold of a space form with constant index of nullity. Then $\nul$ is a $C^{\infty}$ autoparallel distribution of $M$ whose integral manifolds are complete and totally geodesic in the ambient space. \label{null}
\end{lemma}

Under the assumptions of Lemma \ref{null}, denote by $M/\nul$ the set of maximal connected integral manifolds of $\nul$ and let $\pr:M\to M/\nul$ be the  projection to the quotient.

Recall that a coordinate chart $(U,\pp=(x^{1},\cdots,x^{n}))$ of $M$  is called \textsl{plane} for the distribution $\nul$ if for each $a=(a^{1},\cdots,a^{n})\in \pp(U)$, the slice $$S_{a}=\{q\in U:x^{l+1}(q)=a^{l+1},\cdots,x^{n}(q)=a^{n}\}$$ is an integral manifold of $\nul$, where $l$ is the constant index of nullity of $M$.  

From Frobenius theorem, there exists a plane chart $(U,\pp)$ around each point that intersects each integral manifold of $\nul$ in a disjoint countable union of slices of $(U,\pp)$. Since the integral manifolds of $\nul$ are totally geodesic in the ambient space form, it is not difficult to see that there is a plane chart around each point, that intersects each integral manifold of $\nul$  in at most one slice. Then the distribution $\nul$ is regular. 

We give $M/\nul$ the quotient topology and so $\pr$ is open (see \cite{palais}). From the fact that the integral manifolds of $\nul$ are totally geodesic in $M$ one has, as it is not difficult to show, that $\mathcal{R}=\{(x,y)\in M\times M: \pr(x)=\pr(y)\}$ is closed in $M\times M$. This implies that $M/\nul$ is a Hausdorff space (see \cite{kelley}) and therefore it is a differentiable manifold (see \cite[ThmVIII,ChI]{palais}).

\subsection{Fiber bundle structure.}
\label{fibrado}
All throughout this section, $M$ will be a submanifold of the Euclidean space or the sphere. 

 Recall that given three differentiable manifolds $E, N$ and $F$ we say that $\pi:E\to N$ is a \textit{fiber bundle with standard fiber} $F$, if
 \begin{itemize}
 \item[i)] $\pi$ is a $C^{\infty}$ suryective map;
 \item[ii)] there is an open covering $\mathcal{U}$ of $N$ such that for every $U\in\mathcal{U}$ there exists a differentiable map $\Psi:\pi^{-1}(U)\to F$  such that the function $$\overline{\Psi}:=(\pi,\Psi):\pi^{-1}(U)\to U\times F$$ is a diffeomorphism. 
 \end{itemize}
 The map $\overline{\Psi}$ is called a \textsl{local trivialization} for the fiber bundle. 

The projection $\pi$ is a submersion and for every $p\in N$, the fiber $E_{p}=\pi^{-1}(p)$ is an embedded submanifold of $E$ diffeomorphic to $F$. Moreover, if $(U,\overline{\Psi})$ is a local trivialization such that $p\in U$, then the diffeomorphism is given by $\Psi_{|E_{p}}$. 

Consider now the group $Diff(F)$ of diffeomorphisms of $F$ and let $(\pi,\Theta)$, $(\pi,\Psi)$ be local trivializations over open sets $U$ and $V$ on $N$. Then the function $f_{\Theta,\Psi}:U\cap V\to Diff(F)$ given by $f_{\Theta,\Psi}(p)=\Theta\circ\Psi^{-1}_{|E_{p}}$ is called the \textsl{transition function} between both trivializations. The fiber bundle $\pi:E\to N$ is said to have  \textsl{structure group} $G$ if any transition function takes values in a subgroup $G$ of $Diff(G)$. 

Denote by $\an$ the space $\rr^{n}$ with its natural affine structure. A fiber bundle $\pi:E\to N$ with standard fiber $\aff^{n}$ is an \textsl{affine  bundle} if each fiber $E_{p}$ has an affine space structure such that for every local trivialization $(\pi,\Psi)$, $\Psi_{|E_{p}}:E_{p}\to\an$ is an affine isomorphism. Equivalently, $\pi:E\to N$ is an affine bundle if it has structure group $Aff(n):=GL(n)\ltimes \rr^{n}$, the group of affine transformations of $\an$ (analogous to \cite[Prop. 1.14]{poor}). 

A fiber bundle with standard fiber $\esf^{n}$ is called a \textsl{sphere bundle}. Then:

\begin{lemma}
Let $M$ be a complete submanifold of the Euclidean space or the sphere with constant index of nullity $l$ and let $\nul$ be its nullity distribution. Then $\pr:M\to M/\nul$ is an affine bundle if $M$ is a Euclidean submanifold and it is a sphere bundle with structure group $O(l+1)$ if $M$ is a submanifold of the sphere.  
\end{lemma}
\begin{proof}
The proof is standard, but we include it since it is difficult to find in the bliography and we will need some of the notation introduced here.

In order to simplify the notation, we will assume that $M$ is a $1-1$ immersed submanifold. 
Let $(U, \pp=(x^{1},\cdots,x^{n}))$ be a plane regular chart of $M$ with respect to $\nul$, such that $0\in\pp(U)$. Let $l$ be the constant index of nullity of $M$. Then there exists a unique $(n-l)$-dimensional chart $\overline{\pp}$ in $M/\nul$ with domain $\overline{U}=\pr(U)$ such that $\overline{\pp}\circ\pr=(x^{l+1},\cdots,x^{n})$ (see \cite{palais}). Define the local section $\sigma:\overline{U}\to \pr^{-1}(\overline{U})$ by $\sigma(r)=\pp^{-1}(j_{0}(\overline{\pp}(r)))$, where $j_{0}:\rr^{n-l}\to \rr^{n}$ is the inclusion $j_{0}(x^{l+1},\cdots,x^{n})=(0,\cdots,0,x^{l+1},\cdots,x^{n})$. 

Let us begin with the case $M\subset\rr^{n+k}$. Set $M_{r}:=\pr^{-1}(r)$. From Lemma \ref{null}, $M_{r}=p+T_{p}M_{r}$ for any $p\in M_{r}$ (identifying $M_{r}$ with the corresponding subspace of $\rr^{n+k}$). Set $X_{i}(r)=(\partial/\partial x^{i})_{\sigma(r)}\in \rr^{n+k}$ via this identification. So,  if $r\in\overline{U}$ then   any element $x$ of $M_{r}$ is of the form $$x=\sigma(r)+\sum_{i=1}^{l}v^{i}X_{i}(r).$$

For $(r,(v^{1},\cdots,v^{l}))\in\overline{U}\times\aff^{l}$, define  $\rho(r,(v^{1},\cdots,v^{l}))=\sigma(r)+\sum_{i=1}^{l}v^{i}X_{i}(r)$. Then it is not difficult to see that $\rho$ is a diffeormorphism from $\overline{U}\times \aff^{l}$ into $\pr^{-1}(\overline{U})$ and so $\Psi=\rho^{-1}$ is a local trivialization for $\pr:M\to M/\nul$. It is clear from this construction that transition functions are affine.

If $M\subset\esf^{n+k}$. Then $M_{r}$ is the $l$-sphere in the ($l+1$)-dimensional linear subspace 
\begin{equation}
L_{r}:=T_{\sigma(r)}M_{r}\oplus \rr\sigma(r)
\label{Lr}
\end{equation} of $\rr^{n+k+1}$ (regarding $M$ as a submanifold of $\rr^{n+k+1}$ and identifying $T_{\sigma(r)}M_{r}$ with the corresponding subspace of $\rr^{n+k+1}$). Let $\{E_{i}\}$ be the (local) orthonormal frame of $TM$ obtained by applying the Gram-Schmidt orthogonalization process to $\{(\partial/\partial x^{i})\}$. Set $\rho:\overline{U}\times\esf^{l}\to \pr^{-1}(\overline{U})$ as $$(r,(v^{1},\cdots,v^{l+1}))\mapsto  \sum_{i=1}^{l}v^{i}E_{i}(\sigma(r))+v^{l+1}\sigma(r).$$
Then it is not difficult to see that $\rho$ is a diffeormophism and $\Psi=\rho^{-1}$ is a local trivialization for $\pr:M\to M/\nul$. It is clear from construction that the transition functions are in $O(l+1)$.
\end{proof}

Given a piece-wise differentiable curve $\widetilde{c}:I\to M$ we say that $\widetilde{c}$ is \textsl{horizontal with respect to $\nul$} if $\widetilde{c}\;'(t)\;\bot\;\nul_{\widetilde{c}\;(t)}$, for every $t\in I$. Given a curve $c:I\to M/\nul$ we say that $\widetilde{c}:I\to M$ is a \textsl{horizontal lift} of $c$ if $\widetilde{c}$ is horizontal and $\pr(\widetilde{c}(t))=c(t)$ for every $t\in I$. 

As for the case of a vector bundle with a linear connection, one has that any curve in $M/\nul$ can be globally lifted. Namely,

\begin{lemma}
Let $c:I\to M/\nul$ be a (piece-wise differentiable) curve and let $p$ be any point in $\pr^{-1}(c(0))$. Then there is one and only one horizontal lift $\widetilde{c}$ of $c$ to $M$ such that $\widetilde{c}(0)=p$ (called the \emph{horizontal lift} of $c$ through $p$). 
\label{horlift}
\end{lemma}
\begin{proof} Let $c:I\to M/\nul$ be any curve in $M/\nul$. 
It is standard to see, from basic ordinary differential equation theory, that it suffices to prove the following:  for every $b\in I$ there exists an open interval $J_{b}\subset I$ such that $b\in J_{b}$ and such that for every initial condition $z\in \pr^{-1}(c(b))$, there exists a horizontal curve $\widetilde{c}_{b,z}:J_{b}\to M$ which projects down to  $c$ and such that $c_{b,z}(b)=z$. 


 As for any submersion,  we have that for every $z_{0}\in\pr^{-1}(c(b))$ there exists a horizontal lift $\widetilde{c}_{z}$ of $c$ with maximal domain $J_{z_{0}}$ such that $\widetilde{c}_{z}(b)=z$, for every $z$ in a neighborhood of $z_{0}$. 
 
 So, if $M\subset \esf^{n+k}$, the result follows from the fact that the fibers are compact. 
 
 If $M\subset\rr^{n+k}$, let $q_{0},q_{1},\cdots,q_{n}$ form an affine frame on $\pr^{-1}(\gamma(b))$. Set $J_{b}=\cap_{i=0}^{n}J_{q_{i}}$. Then if $z\in pr^{-1}(b)$, $z=q_{0}+\sum_{i=1}^{n}v^{i}(q_{i}-q_{0})$ and it is not difficult to see that $c_{b,z}(t)=c_{q_{0}}(t)+\sum_{i=0}^{n}v^{i}(c_{q_{0}}(t)-c_{q_{i}}(t))$ is the curve we were looking for. 
\end{proof}

We can now define a parallel displacement in $\pr:M\to M\nul$. For $r\in M/\nul$, set $M_{r}:=\pr^{-1}(r)$. Given a (piece-wise differentiable) curve $c:[0,1]\to M/\nul$ we define the $\nul$-\textsl{parallel displacement} $\pe_{c}:M_{c(0)}\to M_{c(1)}$ as $$\pe_{c}(p):=\widetilde{c}_{p}(1)$$
where $\widetilde{c}_{p}$ is the horizontal lift of $c$ through $p$. 

Since every horizontal lift $\widetilde{c}(t)$ of $c$ is perpendicular to the family of totally geodesic submanifolds $M_{c(t)}$,  one has (see e.g. \cite[Lemma 2.8]{olberger})

\begin{lemma}
Let $c:I\to M/\nul$ be a  curve in $M/\nul$. Then the parallel displacement $\pe_{c}:M_{c(0)}\to M_{c(1)}$ is an isometry \qed
\end{lemma}
If $r\in M/\nul$, we denote by $\Omega(r)$ (resp. $\Omega^{0}(r)$) the set of (piece-wise smooth) loops (resp. null-homotopic loops) in $M/\nul$ based at $r$. The \textsl{holonomy group} (associated to $\nul$) based at $r\in M/\nul$ is 
$$\Phi_{r}:=\{\pe_{c}:c\in\Omega(r)\}\subset \text{Iso}(M_{r})$$
and the \textsl{restricted holonomy group} (associated to $\nul$) based at $r$ is the connected subgroup
$$\Phi^{0}_{r}:=\{\pe_{c}:c\in\Omega^{0}(r)\}\subset \text{Iso}_{0}(M_{r}).$$
It is standard to prove, as in the case of a linear connection (see \cite[Teor. 2.25]{poor} or \cite{kn}) that $\Phi_{r}$ and $\Phi^{0}_{r}$ are Lie groups and that $\Phi^{0}_{r}$ is the connected component of the identity in $\Phi_{r}$. 

The \textsl{local holonomy group} (associated to $\nul$) based at $r\in M/\nul$ is defined by 
$$\Phi^{loc}_{r}=\bigcap\Phi^{0}(r,\overline{U})$$
varying $\overline{U}$ among all open neighborhoods of $r$, where $\Phi^{0}(r,\overline{U})=\{\pe_{c}\in\Phi^{0}_{r}:c\subset\overline{U}\}$. One has that there is an open neighborhood $\overline{U}$ of $r$ such that $\Phi^{loc}_{r}=\Phi^{0}(r,\overline{V})$ for every neighborhood $\overline{V}$ of $r$ contained in $\overline{U}$ (see \cite[Prop. 10.1]{kn}).
Let $$(\pe_{c})^{*}(\Phi^{loc}_{r}):=\{(\pe_{c})^{-1}\circ\pe_{\alpha}\circ\pe_{c}:\pe_{\alpha}\in\Phi^{loc}_{r}\}.$$ Then one has the following Ambrose-Singer type theorem, which will be very useful to prove our main global result. 

\begin{lemma}
Let $\mathcal{C}$ be a dense subset of $M/\nul$. Then the restricted holonomy group $\Phi^{0}_{r}$ is generated by the groups $\pe_{c}(\Phi^{loc}_{c(1)})$ varying $c$ among all piece-wise differentiable curves in $M/\nul$ such that $c(0)=r$ and $c(1)\in \mathcal{C}$. 
\label{ambrose}
\end{lemma}

\begin{proof} The proof is similar to the case of a linear connection and we shall indicate its main steps. 

For each $r\in M/\nul$, set $M_{r}=\pr^{-1}(r)$. 

We will start assuming that $M\subset\rr^{n+k}$. Let $A(M_{r})$ denote the set of affine isomorphisms $h:\aff^{l}\to M_{r}$, and set $A(M):=\bigcup_{r\in M/\nul}A(M_{r})$. Let $\pi:A(M)\to M/\nul$ be the canonical projection (i.e, $\pi(h)=r$ if $h\in A(M_{r})$). Then $A(M)$ is a principal fiber bundle with structure group $Aff(l)$ (see \cite[Ch. III, sec. 3]{kn}). 

As for the case of the connection on the frame bundle defined by a vector bundle with a linear connection, one can prove that there is a unique connection $\Gamma$ on $A(M)$ such that the corresponding parallel displacement $\pae$ is related to $\pe$ in the following way (see \cite{poor}). If $h\in A(M_{r})$, then $\{h(0),h(e_{1}),\cdots,h(e_{l})\}$ is an affine frame on $M_{r}$, where $e_{i}$ are the canonical versors in $\rr^{l}$. Let $c:I\to M/\nul$ be a differentiable curve such that $c(0)=r$ and let $q_{0}=\pe_{c}(h(0))$, $q_{i}=\pe_{c}(h(e_{i}))$, then $\pae_{c}(h)$ is the affine isomorphism from $\aff^{l}$ to $M_{c(1)}$ that maps $0$ into $q_{0}$ and $e_{i}$ into $q_{i}$, $i=1,\cdots, l$. 

Given $h\in A(M)$, let $\hol^{0}_{h}$ and $\hol^{loc}_{h}$ denote the restricted and local holonomy groups of $\Gamma$ based at $h$, respectively (recall that $\hol^{0}_{h}$ is the set of elements $g\in Aff(l)$ such that $\pae_{c}(h)=h\circ g$ for some curve $c\in\Omega^{0}(\pi(h))$). If $\pi(h)=r$ then the map $T_{h}(g)=h\circ g\circ h^{-1}$ defines an isomorphism from $\hol^{0}_{h}$ to $\Phi^{0}_{r}$ and from $\hol^{loc}_{h}$ to $\Phi^{loc}_{r}$. 

Now, without almost any modification of the proof of Ambrose-Singer holonomy theorem \cite[Thm. 8.1]{kn} and its consequence \cite[Thm. 10.2]{kn} we can prove that the restricted holonomy group $\hol^{0}_{h}$ is generated by the local holonomy groups $\hol^{loc}_{f}$ varying $f$ in any dense subset $\mathcal{U}$ of the holonomy bundle $P(h)$ (i.e, the set of elements that can be joined to $h$ by a horizontal curve). So $\Phi^{0}_{\pi(h)}$ is generated by the groups $T_{h}(\hol^{loc}_{f})$ varying $f$ as before. Now let $\mathcal{U}=\pi^{-1}(\mathcal{C})\cap P(h)$. If $f\in \mathcal{U}$, then $f=\pae_{c}(h)$ for some $c:I\to M/\nul$ such that $c(0)=\pi(h)$ and $c(1)=\pi(f)\in\mathcal{C}$. From the way $\pae$ and $\pe$ are related it is not difficult to prove that if $g\in \hol^{loc}_{f}$ then $g=f^{-1}\circ\pe_{\alpha}\circ f$ for some $\pe_{\alpha}\in\Phi^{loc}_{\pi(f)}$ and $T_{h}(g)=(\pe_{c})^{-1}\circ \pe_{\alpha}\circ\pe_{c}$. So $T_{h}(\hol^{loc}_{f})=\pe_{c}(\Phi^{loc}_{c(1)})$ as we wanted to prove.

If $M\subset \esf^{n+k}$, consider the principal fiber bundle $A(M):=\bigcup_{r\in M/\nul}A(M_{r})$ with structure group $O(l+1)$ such that $A(M_{r})$ is the set of isometries from $\esf^{l}$ to $M_{r}$. Observe that the elements of $A(M)$ are in a $1-1$ correspondence with the set of orthonormal basis of the subspace $L_{r}$ defined by (\ref{Lr}). The proof follows in the same way as for a submanifold of the Euclidean space. 
\end{proof}

\subsection{Foliating a spherical tube by holonomy tubes}\label{tuboesferico} The technics of this section are mainly inspired on \cite{ocds}.

Let $M$ be a 
submanifold of $\rr^{n}$.
Assume that there exists a positive real number $\varepsilon$ such that the spherical tube $$N:=S_{\varepsilon}(M)=\{q+\xi_{q}:q\in M, \xi_{q}\in\nu_{q}M, \left\|\xi_{q}\right\|=\varepsilon\}$$
is a well defined hypersurface of $\rr^{n}$ (locally this is  always true). 
Consider the canonical projection 
$$
\pi: N\to M,\ \ \ 
 q+\xi_{q} \stackrel{\pi}{\mapsto} q
$$
and the (radial) parallel normal vector field $\Psi$ on $N$ given by 
$$\Psi(x)=\pi(x)-x.$$
Then $M$ is the parallel focal manifold $N_{\Psi}$ to $N$ and $\pi$ is the usual parallel focal  map. 
Since we will work locally, both in $M$ and in $N$, we may also assume that $N$ has constant index of nullity. 

Let $E_{0}=\ker(\hat{A}_{\Psi})$, where by $\hat{A}$ we denote the shape operator of $N$. Then $E_{0}$ is the nullity distribution of $N$, since $N$ is a hypersurface. 

For each $x\in N$, set $$S(x)=\pi^{-1}(\pi(x)).$$
Let $E_{1}=\ker(Id-\hat{A}_{\Psi})=\ker(d\pi)$. Then $E_{1}(x)=T_{x}S(x)$ for every $x\in N$. 

Regard $N$ and $M$ as submanifolds of the ($n+2$)-Lorentzian space $\loren^{n+2}$, identifying $\rr^{n}$ with an $n$-dimensional horosphere of the hyperbolic space $\mathbb{H}^{n+1}$. Denote by $\eta$ the radial normal vector field $\eta(x)=-x$ to $\mathbb{H}^{n+1}$ and set $\widetilde{\Psi}:=\delta\Psi+\eta$, for some small $\delta$ such that $\widetilde{\Psi}$ is timelike. 
Then $$\ker(Id-\hat{A}_{\widetilde{\Psi}})=E_{0}.$$
We can foliate $N$ by the holonomy tubes $$H(x)=(N_{\widetilde{\Psi}})_{-\widetilde{\Psi}(x)}\subset N$$
(locally, see \cite{obc}). We shall further assume that these holonomy tubes have all the same dimension, since we work locally.

If $x,y\in N$, we will denote $x\sim y$ if $x$ and $y$ can be joined by a differential curve in $N$ everywhere normal to $E_{0}$. Then $H(x)$ is locally given by 
$$H(x)=\{y\in N:x\sim y\}$$
(see \cite{ocds}, cf. \cite{hot}).
Since for every $x\in N$, $TS(x)\;\bot\; E_{0}$ we get that $$S(x)\subset H(x).$$

We now consider the distribution $\wnu$ on $N$ given by the normal spaces in $N$ to the holonomy tubes $H(x)$. We will need the following result from \cite[Prop. 2]{ocds}
\begin{lemma} \cite{ocds}
With the above notations, 
\begin{enumerate}
\item the distribution $\wnu$ is autoparallel, invariant under the shape operator of $N$ and contained in the nullity of $N$. Moreover, if $\ws(x)$ is the integral manifold of $\wnu$ through $x$, then $$\ws(x)=(x+\nu_{x}H(x))\cap N.$$ 
\item The restriction of $\wnu$ to any $H(x)$ is a parallel and flat sub-bundle of $\nu_{0}H(x)$, the maximal parallel flat sub-bundle of $\nu H(x)$. Moreover $\ws(y)$ moves parallel in the normal connection of the holonomy tube $H(x)$. \label{ellema1}
\item If $x\in\ws(q)$, then there is a parallel normal field $\varsigma$ to $H(q)$ such that $\varsigma(q)=x-q$ and such that $H(x)=H(q)_{\varsigma}$. \label{ellema2}
\end{enumerate} 
\label{ellema} \qed
\end{lemma}
We aim to prove that $\widetilde{\Sigma}$, and hence $\wnu$, are constant along the fibers of $\pi$. 

Note that $H(x)$ has flat normal bundle, since $$\nu H(x)=\wnu_{|H(x)}\oplus\nu N_{|H(x)}=\wnu_{|H(x)}\oplus\rr\Psi_{|H(x)}.$$

Observe that $\Psi$ can be regarded as the curvature normal associated to the eigenvalue $1$ of the shape operator $\hat{A}_{\Psi}$ of the hypersurface $N$. 
On the other hand, $TN_{|H(x)}=\wnu_{|H(x)}\oplus TH(x)$ with $\wnu$ $\hat{A}$-invariant. Then $H(x)$ is $\hat{A}$-invariant. Finally, since $S(x)\subset H(x)$ we get that  $(E_{1})_{|H(x)}\subset TH(x)$. It then follows that $\overline{\Psi}:=\Psi_{|H(x)}$ is a parallel curvature normal of $H(x)$ for every $x\in N$ (see Lemma $1$ in \cite{ocds}). 

Let us consider $\ws(x)$, the totally geodesic integral manifold of $\wnu$ through $x$. 
Since $\wnu\subset E_{0}$, then  $T\ws(x)$ is contained in the nullity distribution of $N$ and therefore $\ws(x)$ is totally geodesic as a submanifold of $\rr^{n}$ (recall that the integral manifolds of the nullity distribution are totally geodesic in the ambient space). 

Fix $x\in N$ and let $y\in \ws(x)$. Let $\varsigma$ be the parallel normal vector field to $H(x)$ such that $\varsigma(x)=y-x$ given by Lemma \ref{ellema}, (\ref{ellema1}).
Let now $z\in H(x)$. We shall see that $\varsigma$ is constant in the ambient space along $S(z)\subset H(x)$. 

In fact, since $\overline{\Psi}$ and $\varsigma$ are both parallel and $\varsigma$ is tangent to the totally geodesic submanifolds $\ws$, we get that $\left\langle \overline{\Psi}, \varsigma\right\rangle\equiv 0$. On the other hand, given an arbitrary curve $c(t)\subset S(z)$, since $T\ws$ is contained in the nullity of $N$, one has $$\frac{d}{dt}\varsigma(c(t))=A^{H(x)}_{\varsigma}c'(t)=\left\langle \varsigma,\overline{\Psi} \right\rangle c'(t)\equiv 0.$$
So $\varsigma$ is constant along $c$ in the Euclidean ambient space. 

It now follows from Lemma \ref{ellema}, (\ref{ellema1}) that 
if $z\in H(x)$ and $w\in S(z)$, then 
\begin{equation}
\ws(z)=\ws(w)+(z-w).
\label{sigma}
\end{equation}
Observe also that, since $\ws(x)$ is contained in an integral manifold of the nullity distribution of $N$, then the parallel normal field $\Psi$ is constant in the ambient space along $\ws(x)$. So, if $z\in H(x)$, since $z+\varsigma(z)\in \ws(z)$, then 
\begin{equation}
\Psi(z)=\Psi(z+\varsigma(z)).
\label{proj}
\end{equation}

Equation (\ref{sigma}) implies that $\ws$, and hence $\wnu$, are constant along the fibers of $\pi$. And equation (\ref{proj}) implies that $\wnu$ projects down to a well defined distribution on $M$. Therefore we get the following lemma which is standard to prove (cf. \cite[Section 2.5]{ocds}).
\begin{lemma}
The distribution $\wnu$ projects down to a $C^{\infty}$ integrable distribution $\nu^{*}=\pi_{*}(\wnu)$ in $M$ which is autoparallel and contained in the nullity distribution of $M$. If $p\in M$ then $\Sigma^{*}(p)$, the integral manifold of $\nu^{*}$ through $p$, is a totally geodesic submanifold of the Euclidean ambient space and for any $x\in \pi^{-1}(p)$ $$\Sigma^{*}(p)=\ws(x)+\Psi(x).$$ 
The orthogonal complementary distribution $\hor^{*}$ to $\nu^{*}$ in $M$ is integrable, invariant under the shape operators of $M$ and the integral manifold through $p\in M$ locally coincides with  $\pi(H(x))$, for any $x\in\pi^{-1}(p)$. Moreover, the restriction of $\nu^{*}$ to $\pi(H(x))$ is a parallel and flat sub-bundle of the normal space $\nu(\pi(H(x)))$ in the ambient space. 
\label{lema1} \qed
\end{lemma}
\begin{rem}
\emph{
As a consecuence of the previous lemma, one gets that any parallel normal vector field to $H(x)$ (tangent to $N$) projects down to a  parallel normal vector field to  $\pi(H(x))$.}
\label{rem1}
\end{rem}
\begin{rem}
\emph{Observe that since $\nu^{*}\subset \nul$, then $\nul^{\bot}$ is contained in $\hor^{*}$, which is integrable (unlike $\nul^{\bot}$). This will be a key point to prove both the local and the global results. }
\end{rem}

\subsection{Some remarks on polar actions}
Let $Q$ be a space form. A Lie group $G\subset Iso(Q)$ is said to act polarly on $Q$ if there exists a complete, embedded and close submanifold $\Sigma$ that intersects each orbit of $G$ and is perpendicular to orbits at intersection points. The submanifold $\Sigma$ is called a section and it must be totally geodesic. The major property of polar actions is that maximal dimensional orbits are isoparametric submanifolds (see  \cite{palaisterng}, \cite{obc}). A point $p$ such that the orbit $G\cdot p$ is maximal dimensional is called a \textsl{principal point}. 

The following property is very simple to prove and will be very useful in the following sections.

\begin{lemma}
Let $Q$ be a space form and let $G\subset Iso(Q)$ be a Lie group such that its Lie algebra $\mathfrak{g}$ is linearly spanned by the Lie algebras $\mathfrak{g}_{i}$, $i\in I$. Let $G_{i}\subset Iso(Q)$, $i\in I$ be the Lie group associated to the Lie algebra $\mathfrak{g}_{i}$. If each $G_{i}$ acts polarly on $Q$, then the action of $G$ is polar (a section for the $G$-action through a point $p$ is obtained by intersecting the  sections for the $G_{i}$-actions through $p$)\qed

\label{gen}
\end{lemma}



The action of  $G$ is said to be \textsl{locally polar} if the distribution given by the normal spaces to maximal dimensional orbits is integrable (and hence with totally geodesic integral manifolds). The group $G$ may not be closed but the maximal dimensional orbits of an action that is locally polar form a parallel family of isoparametric submanifolds, and so  the closure of $G$ acts polarly on $Q$ and has the same orbits as $G$ (see \cite{palaisterng}, \cite{obc},  \cite{heinliuol}). Therefore we shall make no difference between a polar action and a locally polar action. The following lemma is standard to prove.



\begin{lemma}
Let $G$ be a Lie subgroup of $Iso(Q)$. Assume  there is an open subset $U$ of $Q$ such that the normal spaces of the orbits through points of $U$ define an integrable distribution on $U$. Then the action of $G$ is polar. 
\label{polar}
\end{lemma}

\section{The Local Theorem}
\label{localsection}

\begin{thm}
\label{localth}Let $M$ be a submanifold of the Euclidean space $\rr^{n}$ or the sphere $\esf^{n}$ and let $\nul_{p}$ be the nullity subspace of $M$ at $p\in M$. If $U$ is an open subset of $M$ and $p,q\in U$, denote by $p\sim_{U} q$ if $p$ and $q$ can be joined by a curve $c$ contained in $U$ such that $c'(t)\; \bot\; \nul_{c(t)}$ for every $t$. Let $[p]_{U}$ be the equivalent class $\{q\in U: p\sim_{U} q\}$, for $p\in U$. 

There is an open and dense subset$\ \mathcal{C}$ of $M$ such that for every $p\in \mathcal{C}\ $
one and only one of the following statements holds: 
\begin{enumerate}
\item $[p]_{U}$ contains a neighborhood of $p$, for any open neighborhood $U$ of $p$.
\item There exists an open neighborhood $U$ of $p$ and a proper submanifold $S$ of $U$ such that $p\in S$ and $U$ is the union of parallel manifolds to $S$ over a parallel flat sub-bundle $\nu^{*}$ of $\nu(S)$ contained in $\nul_{U}$. Furthermore, the leaves of this parallel foliation are, locally, equivalent classes. 
\end{enumerate}

\end{thm}
\begin{proof}
We will keep the notations introduced in section \ref{tuboesferico}. 
Assume first that $M$ is a submanifold of the Euclidean space. 
Let $\widetilde{\mathcal{C}}$ be the open and dense subset of $M$ of points $p$ such that the index of nullity $\mu$ is constant around $p$. 
Let $p\in \widetilde{\mathcal{C}}$ and let $U$ be an open neighborhood of $p$ such that there is a well defined spherical tube around $U$. We say that $p$ is a generic point of $\widetilde{\mathcal{C}}$ if $p$ is in the image (via the radial projection) of an open subset of the spherical tube where the index of nullity of the tube is constant and the holonomy tubes $H(x)$ have constan dimension (cf. section \ref{tuboesferico}). 
It is not difficult to see that the set $\mathcal{C}$ of general points of $\widetilde{\mathcal{C}}$ is open and dense in $M$. 

Let $p\in \mathcal{C}$ and assume condition (1) does not hold. Let $U$ be an open neighboorhod of $p$, let $N$ be a spherical tube around $U$ and let $\pi:V\to \pi(V)$ be the radial projection. Let $V$ be an open part of $N$ with constant index of nullity such that $p\in\pi(V)$.  let $\Psi$ be the radial normal vector field defined on $V$ and consider the holonomy tubes $$H(x)=(N_{\widetilde{\Psi}})_{-\widetilde{\Psi}(x)}$$
(cf. section \ref{tuboesferico}), which have constant dimension on $V$. 

Then, from Lemma \ref{lema1}, $U$ is foliated (locally around $p$) by the submanifolds $\pi(H(x))$, which from statemen (3) of Lemma \ref{ellema}, are parallel manifolds over the parallel and flat sub-bundle $\nu^{*}$ of $\nu \pi(H(x))$.

We shall see that if $q=\pi(x)$, then 
\begin{equation}
[q]_{U}=\pi(H(x))\ \text(locally\  around\ q).
\label{eqloc}
\end{equation}

If $y\in H(x)$ near $x$, then there is a curve $\widetilde{c}$ in $V$ joining $x$ and $y$ everywhere perpendicular to $E_{0}=\ker\hat{A}_{\Psi}$, where $\hat{A}$ is the shape operator of $N$. Set $c(t)=\pi(\widetilde{c}(t))=\widetilde{c}(t)+\Psi(\widetilde{c}(t))$. It is not hard to see that $c$ is perpendicular to $(\pi)_{*}(E_{0})$. 
From the tube formula relating the shape operators of $M$ and $N$ (see \cite{obc}), one gets $\nul\subset(\pi)_{*}(E_{0})$. Then $c(t)$ is a horizontal curve (with respect to $\nul$) in $U$ joining $q$ and $\pi(y)$. We conclude that $\pi(H(x))\subset [q]_{U}$ (locally around $q$). 

The other inclusion follows from the fact that the distribution $\hor^{*}$ defined in Lemma \ref{lema1} is integrable and its integral manifolds are the sets $\pi(H(x))$. 
In fact, if $q'\in [q]_{U}$ there is a horizontal curve (with respect to $\nul$) $c(t)$ joining $q$ and $q'$ contained in $U$. But from lemma \ref{lema1}, $\nu^{*}\subset\nul$. Hence $c'(t)\;\bot\;\nu^{*}(c(t))$, that is, $c'(t)\in (\hor^{*})_{c(t)}$, for all $t$. Then $q'$ is in the same integral manifold of $\hor^{*}$ than $q$, so $q'\in\pi(H(x))$.

Suppose now that $M$ is a submanifold of the sphere. Consider the position vector field on $\esf^{n}$ given by $\eta(p)=-p$. We can chose a real positive number $\delta$ small enough such that $$\overline{M}:=\bigcup_{-\delta<\varepsilon<\delta}M_{\varepsilon\eta}$$ is a well defined submanifold of $\rr^{n+1}$, where $M_{\varepsilon\eta}$ is the parallel manifold to $M$ defined by the parallel normal vector field $\varepsilon\eta$. Observe that $M_{\varepsilon\eta}$ is contained in the sphere $\esf_{(1-\varepsilon)}^{n}$ of radius $1-\varepsilon$.

Denote by $\pi_{\varepsilon}:M\to M_{\varepsilon\eta}$ the usual parallel map (observe that $\pi_{0}=Id_{M}$). Let $\overline{\nul}$ be the nullity distribution of $\overline{M}$. Then it is not difficult to see that $$\overline{\nul}_{\pi_{\varepsilon}(p)}=(d\pi_{\varepsilon})_{p}(\nul_{p})\oplus \rr\eta(p)$$

Let $\overline{U}$ be an open set in $\overline{M}$ such that $U:=\overline{U}\cap M$ is open in $M$. Denote by $[q]_{\overline{U}}^{*}$ the set of points in $\overline{M}$ that can be joined to $q$ by a curve perpendicular to $\overline{\nul}$ contained in $\overline{U}$. If $q\in M_{\varepsilon\eta}$, then $[q]^{*}\subset M_{\varepsilon\eta}$ and in particular if $p\in M$, then $[p]_{\overline{U}}^{*}=[p]_{U}$, the set of point of $M$ that can be joined to $p$ by a curve perpendicular to $\nul$ contained in $U$. The theorem for $M$ then follows by applying to $\overline{M}$ the result for submanifolds of the Euclidean space. 
\end{proof}

\section{The Global Theorem}
\label{global}
If $f:M\to\rr^{n}$ is an immersed submanifold, we say that $M$ is \emph{irreducible} if there is no non trivial, $A$-invariant,  autoparallel distribution $\dd$ on $M$ such that $\mathcal{D}^{\bot}$ is also autoparallel. This means that $M$ can not be expressed, locally, as the product of submanifolds of the Euclidean space. If $M$ is complete and simply connected, this is equivalent to the fact that $f$ is not a product of immersions (see \cite{obc}). If $M$ is a submanifold of the sphere, we say that $M$ is irreducible if it is irreducible as a submanifold of the Euclidean space. 

In this section we will prove Theorem \ref{globalthm}. 


Consider the fiber bundle $\pr:M\to M/\nul$ defined in section \ref{fibrado}. We will denote, as in the previous sections, $M_{r}:=\pr^{-1}(r)$ for $r\in M/\nul$.

\begin{lemma}
For any $r\in M/\nul$, the restricted holonomy group $\Phi_{r}^{0}$ acts either transitively or polarly on $M_{r}$. 
\label{accionpolar}
\end{lemma}

\begin{proof}
We will keep the notations of Theorem \ref{localth}. 

Assume first that $M$ is a submanifold of the Euclidean Space. In order to simplify the exposition, we will treat $M$ as an embedded submanifold identifying $i(M)$ with $M$.

First of all, observe that if $\Phi_{r}^{0}$ acts transitively on $M_{r}$ for some $r$, then $\Phi^{0}_{s}$ acts transitively on $M_{s}$ for each $s\in M/\nul$. So let us assume that this action is not transitive. 

Let $\mathcal{C}$ be the open dense subset of $M$ given by Theorem \ref{localth}. If condition (1) holds for some $p$, then $p$ can be joined by a horizontal curve with respect to $\nul$ to any other point in an open neighborhood of $p$ in $M_{r}$ (with $r=\pr(p)$). Then it is not difficult to see that $p$ can be joined to any other point of $M_{r}$ by a horizontal curve, and so the action of $\Phi_{r}^{0}$ on $M$ is transitive. Hence statement (2) on Theorem \ref{localth} must hold for every point of $\mathcal{C}$. 

So let $p\in\mathcal{C}$ and let $U$ be the open neighborhood of $p$ given by theorem \ref{localth}. Set $r=\pr(p)$. We shall see that $$\Phi^{loc}_{r}\cdot q= [q]_{U}\cap M_{r}\ (\text{locally around }q)$$

We can assume (possibly by taking a smaller $U$) that $\Phi^{loc}_{r}=\Phi^{0}(r,\pr(U))$. Then $[q]_{U}\cap M_{r}\subset\Phi^{loc}_{r}\cdot q$.

Now fix a Riemannian metric on $M/\nul$ and define a Riemannian metric on $M$ such that the vertical and horizontal distributions defined by $\pr$ (i.e.,\,$\nul$ and its orthogonal complement) are orthogonal and $\pr$ is a Riemannian submersion. 

For a fixed $\delta>0$  let $P_{\delta}$ be the subset of $\Phi^{loc}_{r}$  consisting of the parallel transport transformations determined by curves in $\pr(U)$ of length less than $\delta$.  Then, following the same ideas of \cite[Appendix]{olmosesch} we obtain that $P_{\delta}$ contains an open neighborhood $\mathcal{U}(r)$ of the identity in $\Phi^{loc}_{r}$. We may take $\delta$  small enough such that $B_{\delta}(q)$ (an open ball for the new metric) is contained in $U$. Then  if $c$ is a loop based at $r$ of length less than $\delta$, its horizontal lift is contained in $U$. So $\mathcal{U}(r)\cdot q$ is an open neighborhood of $q$ in $\Phi^{loc}_{r}\cdot q$ contained in $[q]_{U}$. This proves the other inclusion. 

From statement (2) in Theorem \ref{localth} and Lemma \ref{polar} we conclude that the action of $\Phi_{r}^{loc}$ on $M_{r}$ is polar. 

Let $\widetilde{\mathcal{C}}:=\pr^{-1}(\pr(\mathcal{C}))$.  Then  for every $p\in \widetilde{\mathcal{C}}$, the local holonomy group $\Phi^{loc}_{\pr(p)}$ acts polarly on $M_{\pr(p)}$. The lemma now follows from Lemma \ref{ambrose} and Lemma \ref{gen}

For the case of a submanifold of the sphere, the proof follows in the same way. 
\end{proof}

\begin{proof}[Proof of Theorem \ref{globalthm}]
Since we are working in the cathegory of immersed submanifolds, we may assume that $M$ is simply connected (eventually by passing to the universal cover). 

For $p\in M$ denote by $[p]$ the set of points of $M$ that can be joined to $p$ by a curve horizontal with respect to $\nul$.  Let $V$ be the open and dense subset of principal points of $M$ for the action of the restricted holonomy groups $\Phi^{0}_{r}$. Observe that $V\cap M_{r}$ is also open dense on each fiber $M_{r}$. 

Assume first that $M$ is a submanifold of the Euclidean space. 
If $p\in V$ and $r=\pr(p)$, since the action of $\Phi^{0}_{r}$ is polar by Lemma \ref{accionpolar},  then $\oo p$ is a complete embedded isoparametric submanifold of the Euclidean space $M_{r}$. Therefore $$\oo p=E_{0}(p)\times S(p)$$ where $E_{0}$ is  the nullity subspace of $\oo p$ at $p$ and $S(p)$ is a compact isoparametric submanifold of a sphere (see \cite{palaisterng} or \cite[Theor. 5.2.11]{obc}).

Let $\dd(p)$ be the nullity subspace of $[p]$ at $p$, regarding $[p]$ as a submanifold of $M$ (and not of $\rr^{n+k}$). Set $\hor_{p}:=\nul^{\bot}_{p}\subset T_{p}M$. We shall see that $$\dd(p)=\hor_{p}\oplus E_{0}(p).$$
Fix $p\in V$, $r=\pr(p)$ and let $\xi_{p}\in \nu_{p}(\oo p)\subset \nul_{p}$. 
If $q\in[p]$ and $c$ is a curve in $M/\nul$ such that $q=\pe_{c}(p)$, set $\xi(q):=(d\pe_{c})_{p}(\xi_{p})$. Since $\oo p\subset M_{r}$ is a principal orbit, $\xi$ is a well defined normal vector field to $[p]$ . Moreover, since the action is polar, $\xi$ is parallel in the directions of the orbits of the holonomy groups (cf. \cite{hot} or \cite[Cor. 3.2.5]{obc}). It is also parallel in the directions of $\hor=\nul^{\bot}$. In fact, let $\sigma(t)$ be any horizontal curve contained in $[p]$. Set $p'=\sigma(0)$, $p''=p'+\xi(\sigma(0))$. Observe that the horizontal spaces $\hor$ are constant along any fiber of $\pr:M\to M/\nul$ (since we move along the nullity of $M$).  If $\delta(t)$ is the horizontal lift of $\pr(\sigma)$ through $p''$, then it is not difficult to see that $\xi(\sigma(t))=\delta(t)-\sigma(t)$ and so $\frac{d}{dt}\xi(\sigma(t))=\delta'(t)-\sigma'(t)$, which is horizontal and, in particular, tangent to $[p]$. So $\xi$ is parallel with respect to the normal connection of $[p]$. 

This also proves that $[p+\xi_{p}]$ is the parallel (possibly focal) manifold $[p]_{\xi}$ to $[p]$. 

Let $q=p+\xi_{p}$ and so $\hor_{q}$ and $\hor_{p}$ coincide (as subspaces of $\rr^{n+k}$). Since they are both isomorphic to $T_{\pr(p)}M/\nul$ via $d\pr$, one has the isomorphism  $$\pp:=d\pr_{q}^{-1}\circ (d\pr_{p})_{|\hor_{p}}:\hor_{p}\to\hor_{q}\simeq\hor_{p}.$$ Let $X\in T_{\pr(p)} M/\nul$ and $c(t)$ a curve in $M/\nul$ such that $c(0)=\pr(p)$ and $c'(0)=X$. Let $\sigma(t)$ and $\beta(t)$ be the horizontal lifts of $c$ through $p$ and $q$ respectively.  We have seen that $\beta(t)=\sigma(t)+\xi(\sigma(t))$. So 
$$\beta'(0)=\sigma'(0)+\frac{d}{dt}\xi(\sigma(t))=(Id-\widehat{A}_{\xi_{p}})\sigma'(0)$$
where $\widehat{A}_{\xi_{p}}$ is the shape operator of $[p]$ as a submanifold of $M$ (which coincides with the shape operator as a submanifold of $\rr^{n+k}$). 

Hence $\pp=(Id-\widehat{A}_{\xi_{p}})$ is an isomorphism from  $\hor_{p}$ to $\hor_{q}\simeq\hor_{p}$ for each $\xi_{p}\in\nu_{p}(\oo p)$. Suppose now that there exists a normal vector $\xi_{p}$ to the orbit $\oo p$ such that $\widehat{A}_{\xi_{p}|\hor_{p}}\neq0$. Then there is an eigenvector $v\in\hor_{p}$ associated to a real eigenvalue $\lambda\neq 0$ of $\hat{A}$. Then $(Id-\widehat{A}_{\xi_{p}/\lambda})v=0$, which can not occur. So $\hor_{p}$ is contained in the nullity subspace of $[p]$ at $p$. We conclude that $\mathcal{D}(p)=E_{0}(p)\oplus \hor_{p}$ as we wanted to show. 

Recall that $\mathcal{D}(p)$ is defined only on the dense subset $V$. 

Now observe that since any two maximal dimensional orbits on the same fiber $M_{r}$ are parallel manifolds, they have the same extrinsic Euclidean factor (regarded as submanifolds of $M_{r}$). Therefore, the subspaces $E_{0}(p)$ can be all identified on $V\cap M_{r}$ for every fiber $M_{r}$. We can hence extend the distribution $E_{0}$ to the hole $M_{r}$ define $E_{0}(q)$ for $q\in M_{r}$ as the common subspace $E_{0}(p)$ for any $p\in M_{r}\cap V$. If $p$ and $q$ are in different integral manifolds, then there is a point $q'\in M_{\pr(q)}$ such that $\oo q'$ is the parallel translated of $\oo p$ (along an appropriate curve in $M/\nul$) and is therefore isometric to it. So $dim(E_{0}(q))=dim(E_{0}(q'))=dim(E_{0}(p))$ and $\mathcal{D}$ is a well defined differential distribution on $M$. 

We will prove that $\mathcal{D}$ and $\mathcal{D}^{\bot}$ are autoparallel and invariant under the shape operators $A$ of $M$. 

Observe that $\mathcal{D}^{\bot}(p)$ in $M$ is the orthogonal complement of $E_{0}(p)$ in $M_{\pr(p)}$ and so $\mathcal{D}^{\bot}$ is an autoparallel distribution on $M$, which is parallel when restricted to any fiber $M_{r}$. Since $\hor$ is $A$-invariant and $A_{|E_{0}}\equiv 0$, we get that $\mathcal{D}$ is $A$-invariant. Denote by $\nu^{M}[p]$ and by $\nu^{\rr^{n+k}}[p]$ the normal bundles of $[p]$ as a submanifold of $M$ or $\rr^{n+k}$ respectively. Since $[p]$ is $A$-invariant, then $\nu^{M}[p]$ is a parallel sub-bundle of $\nu^{\rr^{n+k}}[p]$. From Codazzi equation, the  distribution $\dd$  (which is the nullity of $[p]$ associated to a parallel sub-bundle of $\nu^{\rr^{n+k}}[p]$) is autoparallel in $\mathcal{V}$. 
 Since $V$ is dense in $M$, we get that $\dd$ is autoparallel. 

Since $M$ is irreducible, $\dd^{\bot}$ must be trivial. So any orbit $\oo p$ coincides with the whole integral manifold $M_{\pr(p)}$. 

Assume now that $M$ is a submanifold of the sphere $\esf^{n+k}$. 
As in the proof of Theorem \ref{localth}, let $\eta(p)=-p$ be the position normal vector field and set  $N:=\bigcup_{-\delta<\varepsilon<\delta}M_{\varepsilon\eta}$. 

If $r\in M/\nul$, let $L_{r}:=T_{p}M_{r}\oplus \rr\eta_{p}$. $L_{p}$ is the smallest linear subspace of $\rr^{n+k+1}$ that contains the sphere $M_{r}$. 

Let $p\in V$, where $V$ is as in the Euclidean case. Suppose there exists a  normal vector $0\neq\xi_{p}\in \nu_{p}[p]$, where $\nu_{p}[p]$ is the normal space of $[p]$ as a submanifold of $M$. Then $\xi_{p}\in L_{\pr(p)}$ is normal to the isoparametric orbit $\Phi^{0}_{r}\cdot p$. We can extend $\xi_{p}$ to a normal parallel vector field $\xi$ to $[p]$ (as in the Euclidean case).

Consider in $\rr^{n+k+1}$ the parallel manifold $\widetilde{[p]}:=[p]_{\xi}$ and then consider the projection $\widetilde{[p]}_{\lambda\eta}$ of this submanifold to the sphere, where $\lambda\neq 1$ is a suitable real number. Then $[p]$ and $\widetilde{[p]}_{\lambda\eta}$ are parallel manifolds on the sphere $\esf^{n+k}$. Let $\pi_{\xi}:[p]\to \widetilde{[p]}$ and $\pi_{\lambda\eta}:\widetilde{[p]}\to \widetilde{[p]}_{\lambda\eta}$ be the corresponding focal maps. Then $\pi_{\lambda\eta}\circ\pi_{\xi}(p)\in M_{\pr(p)}$.

Observe that the orthogonal complement $\hor_{p}$ to $\nul_{p}$ in $T_{p}M$ is constant along $M_{\pr(p)}$. In fact, if we consider the nullity $\overline{\nul}$ of $N$, then $\hor$ is also the horizontal space associated to $\overline{\nul}$ and therefore constant along $L_{\pr(p)}$. 

Then it is not difficult to see that the isomorphism $d\pr_{q}^{-1}\circ d\pr_{p}$ from $\hor_{p}$ to $\hor_{q}\simeq\hor_{p}$ is given by $d\pi_{\lambda\eta}\circ (Id-\hat{A}_{\xi})_{|\hor_{p}}$, where $\hat{A}$ is the shape operator of $[p]$.
In the same way as in the Euclidean case, we conclude that $\hat{A}_{\xi|\hor_{p}}\equiv 0$. 

Let $\dd(p)$ be the nullity subspace of $[p]$ at $p$ as a submanifold of $M$. Then $\hor_{p}\subset\dd(p)$. Since the orbits of $\Phi^{0}_{r}$ through points in $V$ are isoparametric submanifolds of a sphere we may assume that they have no nullity in the sphere (eventually by passing to a nearby parallel orbit).

Then $\hor_{|V}=\dd$ is an autoparallel distribution on $V$, and since $V$ is dense, $\hor$ is autoparallel in $M$. 
Since $M$ is irreducible, and both $\hor$ and $\nul$ are non trivial, the normal space to any orbit $\Phi^{0}_{\pr(p)}\cdot p$ is trivial, or equivalently the orbit coincides with $M_{r}$.
\end{proof}

\begin{rem} \emph{It is possible to show with a simple example, that the global theorem is false for submanifolds of the hyperbolic space. We will construct a $2$-dimensional $1-1$ immersed complete submanifold of the hyperbolic space $\hip^{3}$ with constant index of nullity $1$ (and such that the perpendicular distribution to the nullity is hence integrable). This submanifold is given as a union of orbits of points in a geodesic by a $1$-parameter subgroup of isometries.}

\noindent
\emph{Regard $\hip^{3}$ as the connected component through $e_{4}=(0,0,0,1)$ of  $\{x\in\loren^{4}: \meti x,x\metd=-1\}$ where $\loren^{4}$ is the space $\rr^{4}$ with the Lorentz metric $\meti x,y\metd=x_{1}y_{1}+x_{2}y_{2}+x_{3}y_{3}-x_{4}y_{4}$. Let $\sigma$ be the geodesic trough $e_{4}$ in $\hip^{3}$ such that $\sigma'(0)=v$, that is, $\sigma(s)=\sinh(s)v+\cosh(s)e_{4}$, and let $w\in T_{e_{4}}\hip^{3}$ such that $w\; \bot\; v$. Consider a matrix $B\in \mathfrak{so}(3)$ such that $Bv=0$ and $Bw\neq0$, and such that $A=\left(
\begin{array}{cc}
B& w\\
w^{\mathbf{t}}& 0\\
\end{array}\right)$ verifies $A^{3}=0$ (it is not difficult to see that it is always possible). Observe that $A\in\mathfrak{so}_{1}(4)$, the Lie algebra of the group of isometries of $\hip^{3}$. Let $X$ be the Killing vector field on $\hip^{3}$ defined by $A$ and let $\{\pp_{t}\}$ the one parameter subgroup associated to it. 
 Define the function $$f:\rr^{2}\to\hip^{3}\; /\; (t,s)\mapsto f(t,s)=\pp_{t}(\sigma(s))=e^{tA}\sigma(s).$$
 Then it is not hard to see that $f$ is a $1-1$ immersion. Let $M=f(\rr^{2})\subset\hip^{3}$. Then it is not difficult to prove that 
 \begin{enumerate}
 \item[i)] The nullity subspace of $M$ at $p=f(s,t)$ is generated by the tangent vector $(d\pp_{t})_{\sigma(s)}\sigma'(s)$. Therefore $M$ has constant index of nullity $1$. 
 \item[ii)] If $M$ were an extrinsic product in the Lorentz space $\loren^{4}$, it would have constant sectional curvature equal to $0$. But since it is a surface in $\hip^{3}$ with nullity, it has constant curvature $-1$. Therefore $M$ is irreducible. 
 \item[iii)] Since $A^{3}=0$, $\pp_{t}=\text{e}^{tA}=I+tA+\frac{t^{2}}{2}A^{2}$. Then it is not difficult to show that any Cauchy sequence on $M$ is convergent, and therefore $M$ is complete 
 \end{enumerate}
\label{hyp}
This procedure can be generalized to higher dimensions, by asking further properties to the matrix $A$. }
 \end{rem}

\textbf{Acknowledgement}: This work is part of the author's Ph.D. thesis, written in FCEIA, UNR, directed of Prof. Carlos Olmos. 

\bibliographystyle{alpha}
\bibliography{Articulo}

\begin{thebibliography}{CDO09}

\bibitem[BCO03]{obc}
J.~Berndt, S.~Console, and C.~Olmos.
\newblock {\em Submanifolds and holonomy}.
\newblock Chapman and Hall. Boca Raton, 2003.

\bibitem[CDO09]{ocds}
S.~Console, A.~DiScala, and C.~Olmos.
\newblock A berger type normal holonomy theorem for complex submanifolds.
\newblock {\em To appear in Math. Ann., ArXiv:math.DG0807.3419v2}, 2009.

\bibitem[Daj90]{daj}
M.~Dajczer.
\newblock {\em Submanifolds and Isometric Immersions}.
\newblock Math. Lecture Series 13. Publish or Perish, 1990.

\bibitem[DG85]{dajgrom}
M.~Dajczer and D.~Gromoll.
\newblock Gauss parametrizations and rigidity aspects of submanifolds.
\newblock {\em J. of Diff. Geom.}, 22:1--12, 1985.

\bibitem[EO94]{olmosesch}
J.~H Eschenburg and C.~Olmos.
\newblock Rank and symmetry of riemannian manifolds.
\newblock {\em Comment. Math. Helv.}, 69:483--499, 1994.

\bibitem[Fer71]{ferus}
D.~Ferus.
\newblock On the completeness of nullity foliations.
\newblock {\em Michigan Math. J.}, 18:61--64, 1971.

\bibitem[Gro99]{gromov}
M.~Gromov.
\newblock {\em Metric structures for Riemannian and non-Riemannian spaces}.
\newblock Progr. Math.,152. Birkhäuser Boston, 1999.

\bibitem[HL99]{Heintze}
E.~Heintze and X.~Liu.
\newblock Homogeneity of infinite dimensional isoparametric submanifolds.
\newblock {\em Annals of Math.}, 149(2):149--181, 1999.

\bibitem[HLO06]{heinliuol}
E.~Heintze, X.~Liu, and C.~Olmos.
\newblock Isoparametric submanifolds and a chevalley-type restriction theorem.
\newblock {\em Integrable systems, geometry and topology, AMS}, pages 151--190,
  2006.

\bibitem[HOT91]{hot}
E.~Heintze, C.~Olmos, and G.~Thorbergsson.
\newblock Submanifolds with constant principal curvatures and normal holonomy
  groups.
\newblock {\em Int. J. Math Vol 2 No 2}, 2:167--175, 1991.

\bibitem[Kel75]{kelley}
J.~L. Kelley.
\newblock {\em General Topology}.
\newblock Graduate Texts in Mathematics. Springer, 1975.

\bibitem[KN63]{kn}
S.~Kobayashi and K.~Nomizu.
\newblock {\em Foundations of differential geometry}, volume~1.
\newblock John Wiley and Sons, 1963.

\bibitem[Olm05]{olberger}
C.~Olmos.
\newblock A geometric proof of the berger holonomy theorem.
\newblock {\em Annals of Math.}, 161, 2005.

\bibitem[Pal57]{palais}
R.~Palais.
\newblock {\em A global formulation of the Lie theory of transformation
  groups}.
\newblock Mem. Amer. Math. Soc. 1957.

\bibitem[Poo81]{poor}
W.~A. Poor.
\newblock {\em Differential Geometric Structures}.
\newblock McGraw-Hill, 1981.

\bibitem[PT88]{palaisterng}
R.~Palais and C.L. Terng.
\newblock {\em Critical point theory and submanifold geometry}.
\newblock Lecture Notes in Mathematics. Springer-Verlag, 1988.

\bibitem[Tho91]{tho}
G.~Thorbergsson.
\newblock Isoparametric foliations and their buildings.
\newblock {\em Annals of Math.}, 133(2):429--446, 1991.

\end{thebibliography}

\end{document}